\documentclass[a4paper,11pt, leqno]{article}
\usepackage[latin1]{inputenc}

\usepackage[intlimits]{amsmath}
\usepackage{amssymb}
\usepackage{amsfonts}
\usepackage{mathrsfs}
\usepackage{amsthm}
\usepackage{colonequals,comment}
\usepackage{mathrsfs}
\usepackage[all]{xy}

\RequirePackage{ifthen}
\RequirePackage{verbatim}

\newskip\procskipamount
\newskip\interskipamount
\newskip\refskipamount
\newcommand{\procskip}{\vskip\procskipamount}
\newcommand{\interskip}{\vskip\interskipamount}
\newcommand{\refskip}{\vskip\refskipamount}

\newcommand{\procbreak}{\par
   \ifdim\lastskip<\procskipamount\removelastskip
   \penalty-100
   \procskip\fi
   \noindent\ignorespaces}

\newcommand{\titlebreak}{\par%
\ifdim\lastskip<\interskipamount\removelastskip%
\penalty10000%
\interskip\fi%
\noindent}%

\newcommand{\interbreak}{\par%
\ifdim\lastskip<\interskipamount\removelastskip%
\penalty-100%
\interskip\fi%
\noindent\ignorespaces}%

\newcommand{\refbreak}{\par%
\ifdim\lastskip<\refskipamount\removelastskip%
\penalty-100%
\refskip\fi%
\noindent\ignorespaces}%

\theoremstyle{plain}
\newtheorem{theorem}{Theorem}[section]
\newtheorem{lemma}[theorem]{Lemma}
\newtheorem{corollary}[theorem]{Corollary}
\newtheorem{proposition}[theorem]{Proposition}

\newtheorem{intro-theorem}{Theorem}
\newtheorem{intro-corollary}[intro-theorem]{Corollary}
\newtheorem{intro-proposition}[intro-theorem]{Proposition}

\theoremstyle{definition}

\newtheorem{definition}[theorem]{Definition}

\newtheorem{remark}[theorem]{Remark}

\newtheorem{example}[theorem]{Example}

\numberwithin{equation}{subsection}

\newcommand{\marginrule}{\marginpar{\rule[-10.5mm]{1mm}{10mm}}}

\newcounter{refcounter}

{\par}%

\newcounter{listcounter}
\newcounter{deflistcounter}
\newcounter{equivcounter}

\newskip{\itemsepamount}
\newskip{\topsepamount}
\newenvironment{assertionlist}{%
  \begin{list}
    {\upshape (\arabic{listcounter})}
    {\setlength{\leftmargin}{18pt}
     \setlength{\rightmargin}{0pt}
     \setlength{\itemindent}{0pt}
     \setlength{\labelsep}{5pt}
     \setlength{\labelwidth}{13pt}
     \setlength{\listparindent}{\parindent}
     \setlength{\parsep}{0pt}
     \setlength{\itemsep}{\itemsepamount}
     \setlength{\topsep}{\topsepamount}
     \usecounter{listcounter}}}
  {\end{list}}

\newenvironment{definitionlist}{%
  \begin{list}
    {\upshape (\alph{deflistcounter})}
    {\setlength{\leftmargin}{18pt}
     \setlength{\rightmargin}{0pt}
     \setlength{\itemindent}{0pt}
     \setlength{\labelsep}{5pt}
     \setlength{\labelwidth}{13pt}
     \setlength{\listparindent}{\parindent}
     \setlength{\parsep}{0pt}
     \setlength{\itemsep}{\itemsepamount}
     \setlength{\topsep}{\topsepamount}
     \usecounter{deflistcounter}}}
  {\end{list}}

\newenvironment{equivlist}{%
  \begin{list}
    {\upshape (\roman{equivcounter})}
    {\setlength{\leftmargin}{18pt}
     \setlength{\rightmargin}{0pt}
     \setlength{\itemindent}{0pt}
     \setlength{\labelsep}{5pt}
     \setlength{\labelwidth}{13pt}
     \setlength{\listparindent}{\parindent}
     \setlength{\parsep}{0pt}
     \setlength{\itemsep}{\itemsepamount}
     \setlength{\topsep}{\topsepamount}
     \usecounter{equivcounter}}}
  {\end{list}}

% ...

\renewcommand{\AA}{\mathbb{A}}
\newcommand{\BB}{\mathbb{B}}
\newcommand{\CC}{\mathbb{C}}
\newcommand{\DD}{\mathbb{D}}

\newcommand{\FF}{\mathbb{F}}
\newcommand{\GG}{\mathbb{G}}

\newcommand{\II}{\mathbb{I}}

\newcommand{\QQ}{\mathbb{Q}}

\newcommand{\WW}{\mathbb{W}}

\newcommand{\ZZ}{\mathbb{Z}}

\renewcommand{\hbar}{\bar{h}}

\newcommand{\kgbar}{{\bar{\kappa}}}

\newcommand{\FFbar}{\overline{\FF}}

\newcommand{\QQbar}{\overline{\QQ}}

\newcommand{\Tbf}{{\bf T}}

\newcommand{\Acal}{{\mathcal A}}
\newcommand{\Bcal}{{\mathcal B}}

\newcommand{\Ncal}{{\mathcal N}}

\newcommand{\Pcal}{{\mathcal P}}

\newcommand{\Ascr}{{\mathscr A}}

\newcommand{\Dscr}{{\mathscr D}}

\newcommand{\Gscr}{{\mathscr G}}
\newcommand{\Hscr}{{\mathscr H}}

\newcommand{\Oscr}{{\mathscr O}}

\newcommand{\Xscr}{{\mathscr X}}

\newcommand{\gtilde}{\tilde{g}}

\newcommand{\xtilde}{\tilde{x}}

\newcommand{\Ctilde}{\tilde{C}}
\newcommand{\Dtilde}{\tilde{D}}

\newcommand{\Ftilde}{\tilde{F}}
\newcommand{\Gtilde}{\tilde{G}}

\newcommand{\Ptilde}{\tilde{P}}
\newcommand{\Qtilde}{\tilde{Q}}

\newcommand{\Ttilde}{\tilde{T}}

\newcommand{\eps}{\varepsilon}

\DeclareMathOperator{\DR}{DR}

\DeclareMathOperator{\Frob}{Frob}
\DeclareMathOperator{\Gal}{Gal}

\DeclareMathOperator{\GSp}{GSp}

\DeclareMathOperator{\Hom}{Hom}

\DeclareMathOperator{\Ker}{Ker}

\renewcommand{\mod}[1]{\allowbreak \mkern5mu {\rm mod}\,\,#1}

\newcommand{\One}{\hbox{\rm1\kern-2.3ptl}}

\DeclareMathOperator{\Par}{Par}

\DeclareMathOperator{\Spec}{Spec}

\DeclareMathOperator{\Stab}{Stab}

\newcommand\addots{\mathinner{\mkern1mu\raise0pt\vbox{\kern7pt\hbox{.}}\mkern2mu\raise3pt\hbox{.}\mkern2mu\raise6pt\hbox{.}\mkern1mu}}
\newcommand{\bs}{\backslash}

\newcommand{\doubleexp}[3]{\leftexp{#1}{#2}{\vphantom{#2}}^{#3}}

\newcommand\eqann[1]{\mathrel{\smash{\overset{\ #1\ }{=}}}}

\newcommand{\leftexp}[2]{{\vphantom{#2}}^{#1}{#2}}

\newcommand{\lrangle}{{\langle\ ,\ \rangle}}
\renewcommand{\mod}[1]{\allowbreak \mkern5mu {\rm mod}\,\,#1}

\newcommand{\restricted}[1]{{}_{\vert}{}_{#1}}
\newcommand{\rstr}[1]{\restricted{#1}}
\newcommand{\set}[2]{\{\,#1\ ;\  #2\,\}}

\renewcommand{\star}{{}^*}

\newcommand{\vdual}{{}^{\vee}}

\newcommand\varto[1]{\mathrel{\hbox to #1pt{\rightarrowfill}}}
\newcommand\vartoover[2]{\mathrel{\smash{\overset{#2}{\varto{#1}}}}}

\newcommand{\mono}{\hookrightarrow}

\newcommand{\sends}{\mapsto}

\newcommand{\iso}{\overset{\sim}{\to}}

\newcommand{\APlus}{(A,\iota,\lambda,\eta)}

\newcommand{\AStrat}[1]{{\vphantom{\Ascr}}^{#1}\negmedspace{\Ascr}}

\newcommand{\DMod}{\Dscr{-}{\tt Mod}}

\numberwithin{equation}{section}

\begin{document}

\title{Bruhat strata for Shimura varieties of PEL type}

\author{Torsten Wedhorn\footnote{Dept. of Mathematics, University of Paderborn, Warburger Str. 100, D-33098 Paderborn, Germany, {\tt wedhorn@math.uni-paderborn.de}}}

\maketitle


\noindent{\scshape Abstract.\ } We study the Bruhat stratification for Shimura varieties of PEL type. In the Siegel case this stratification is a scheme-theoretic variant of the stratification by the $a$-number. We show that all Bruhat strata are smooth and determine their dimensions. We also prove that the closure of a Bruhat stratum is a union of Bruhat strata and describe which Bruhat strata are contained in the closure of a given Bruhat stratum.

\medskip




\section*{Introduction}

\subsection*{Background}\label{background}

Let $X$ be an abelian variety over a perfect field $k$ of characteristic $p > 0$. An important invariant of $X$ is its $a$-number $a(X) := \dim_{k}\Hom(\alpha_p,X)$, where $\alpha_p := \Ker(\Frob\colon \GG_a \to \GG_a)$. One always has $0 \leq a(X) \leq \dim(X)$ with $a(X) = 0$ (resp.~$a(X) = \dim(X)$) if and only if $X$ is an ordinary abelian variety (resp.~$X$ is a superspecial abelian variety). The stratification by $a$-number on moduli spaces of polarized abelian varieties has been studied intensively for instance by G.~van der Geer in \cite{vdG_CyclesAV}. F.~Oort has shown that in every non-ordinary Newton stratum of the moduli spaces of principally polarized abelian varieties the stratum where the $a$-number is equal to $1$ is open and dense in that Newton stratum. He then used this fact to prove a conjecture by Grothendieck which says that the closure of Newton stratum corresponding to a (concave) Newton polygon $\nu$ is the union of those Newton strata corresponding to the Newton polygons lying below $\nu$ (\cite{Oort3}).

\subsection*{Main results}

In this paper we study the Bruhat stratification introduced in \cite{Wd_Bruhat} for good reductions of Shimura varieties of PEL type. These good reductions are still moduli spaces of abelian varieties with additional structures. We show in Example~\ref{Siegel} that for the Siegel case (i.e. for the moduli space of principally polarized abelian varieties in characterstic $p > 0$) the Bruhat stratification is a scheme-theoretic variant of the stratification by the $a$-number.

In general, the Bruhat stratification is a decomposition of the special fiber $\Ascr_0$ of the PEL-moduli space into locally closed subspaces. These subspaces can roughly be described as follows. The first de Rham homology $H_1^{\rm DR}(\Xscr/\Ascr_0)$ of the universal abelian scheme $\Xscr$ over $\Ascr_0$ with all its additional structures is endowed with two local direct summands, the Hodge filtration and the conjugate filtration (see \eqref{HodgeFil} for a precise definition). The Bruhat strata are then the loci, where these two filtrations with all their additional structures are in a fixed relative position. In the case of the moduli space of principally polarized abelian varieties of dimension $g$, $H_1^{\rm DR}(\Xscr/\Ascr_0)$ is a locally free module of rank $2g$ endowed with a symplectic pairing given by the polarization, and the Hodge $C$ and the conjugate filtration $D$ are totally isotropic local direct summands of $H_1^{\rm DR}(\Xscr/\Ascr_0)$ of rank $g$. Then the Bruhat strata are simply the locally closed subspaces $\AStrat{a}$, where $C + D$ is a local direct summand of $H_1^{\rm DR}(\Xscr/\Ascr_0)$ of some fixed rank $2g - a$. Moreover $\AStrat{a}(\FFbar_p)$ consists of those principally polarized abelian varieties $(X,\lambda)$ over $\FFbar_p$ with $a(X) = a$.

In general, the Bruhat stratification is indexed as follows. Let $\GG$ be the reductive group over $\QQ$ of the PEL-Shimura datum. As we assume that the Shimura variety has good reduction at $p$, the group $\GG_{\QQ_p}$ has a reductive model $\Gtilde$ over $\ZZ_p$. Let $G/\FF_p$ be its special fiber and denote by $(W,I)$ its Weyl group together with its set of simple reflections. Let $[\mu]$ be the minuscule conjugacy class of cocharacters of $G_{\FFbar_p}$ given by the Shimura datum. It determines a conjugacy class of parabolic subgroups of $G_{\FFbar_p}$ and hence a subset $J \subseteq I$ (see~\eqref{DefineJ} for its precise definition). The (geometric) Frobenius acts on $(W,I)$ via an automorphism $\bar\varphi$ of Coxeter systems. We set $K := \bar\varphi(J)^{\rm opp}$, where $(\ )^{\rm opp}$ denotes the opposite type. The field of definition of $J$ is the finite extension $\kappa$ of $\FF_p$ over which the special fiber $\Ascr_0$ of the PEL-moduli space is defined. Let $\Gamma_J$ be the Galois group of $\kappa$.

We let $\doubleexp{J}{W}{K}$ be the set of elements in $w \in W$ that are of minimal length in their double coset $W_JwW_K$, where $W_J$ and $W_K$ are the subgroups of $W$ generated by $J$ resp.~$K$. Then the Galois group $\Gamma_J$ acts on $\doubleexp{J}{W}{K}$ and the Bruhat stratification is a decomposition
\[
\Ascr_0 = \bigcup_{[x] \in \Gamma_J\bs\doubleexp{J}{W}{K}}\AStrat{[x]}
\]
into locally closed subspaces. Our main result is the following (Corollary~\ref{AStratProp}):

\begin{intro-theorem}\label{MainThm1}
The Bruhat strata $\AStrat{[x]}$ are smooth of dimension $\ell(x^{J,K})$ (see \eqref{DefEllJK} for the definition of $x^{J,K}$). In particular all Bruhat strata are non-empty. The closure of a Bruhat stratum is given by
\[
\overline{\AStrat{[x]}} = \bigcup_{[x'] \leq [x]}\AStrat{[x']},
\]
where $\leq$ denotes the partial order on $\Gamma_J\bs\doubleexp{J}{W}{K}$ induced by the Bruhat order on the Coxeter group $W$.
\end{intro-theorem}

In the Siegel case we show in Example~\ref{BruhatSiegel} that one reobtains the known formulas for the dimension of the $a$-number strata (e.g., see \cite{vdG_CyclesAV}).

There are two essential tools to prove Theorem~\ref{MainThm1}. The first is the description of the Bruhat stack $\Bcal_{J,K}$ \eqref{DefBruhatStack} obtained in \cite{Wd_Bruhat} and briefly recalled in the beginning of Section~2. The second tool is to show that the morphism
\[
a\colon \Ascr_0 \to \Bcal_{J,K},
\]
which sends an abelian variety $X$ with additional structure to the triple consisting of $H_1^{\rm DR}(X)$ (with its additional structure), the Hodge filtration, and the conjugate filtration, is a smooth morphism (Theorem~\ref{ASmooth}).

In \cite{ViWd} Viehmann and the author defined and studied the Ekedahl-Oort stratification for Shimura varieties of PEL type. By construction every Bruhat stratum is a union of Ekedahl-Oort strata. Using the results of \cite{Wd_Bruhat} we describe in Proposition~\ref{EOaStrata} which Ekedahl-Oort strata are contained in a given Bruhat stratum. We deduce the following result (Theorem~\ref{MuOrdAStrat}):

\begin{intro-theorem}\label{MainThm2}
There exists a unique open Bruhat stratum $\AStrat{[\xtilde]}$. It is dense in $\Ascr_0$. Moreover the following assertions are equivalent.
\begin{equivlist}
\item
$\AStrat{[\xtilde]}$ is equal to the generic Newton stratum.
\item
The field of definition $\kappa$ of $J$ is equal to $\FF_p$.
\item
There exists an abelian variety $X$ with additional structure in $\Ascr_0(\FFbar_p)$ which is ordinary.
\end{equivlist}
\end{intro-theorem}


\subsection*{Contents of the paper}

In Section~\ref{PEL} we introduce the PEL-moduli spaces and all other notations. We also very briefly recall the Newton stratification. In Section~\ref{BRUHAT} we define the Bruhat stratification of the special fiber of the PEL-Moduli space and show that it generalizes the $a$-number stratification in the Siegel case.

We compare the Ekedahl-Oort and the Bruhat stratification and deduce Theorem~\ref{MainThm2} in Section~\ref{EOAndAStrat}. Section~\ref{PROP} contains Theorem~\ref{MainThm1} and its proof using a general result about deformations of opposite parabolic subgroups. This result is shown in Section~\ref{PAROPP}.

\noindent{\scshape Acknowledgements}.\ I am grateful to the referee for his/her careful reading and for the suggestion of Remark~\ref{MaxGeom}.


\section{Moduli spaces of PEL type with good reduction}\label{PEL}
Let $S$ be a scheme over~$\FF_p$ and let $q$ be a power of $p$. The pullback of a scheme or a sheaf or a morphism over $S$ under the $q^{\rm th}$-power Frobenius morphism $S \to S$ is denoted by $(\ )^{(q)}$. 

In this section we recall the notion of Shimura-PEL-data and their attached moduli spaces. Our main references are Kottwitz~\cite{Ko_ShFin} and Rapoport and Zink~\cite{RZ_Period}. 

\subsection*{The PEL moduli space $\Ascr$}
Let $\Dscr = \bigl(\BB,\star,V,\lrangle, O_{\BB}, \Lambda, h\bigr)$ denote an integral Shimura-PEL-datum that is unramified at a prime~$p > 0$ in the sense of~\cite{ViWd}~Section~1.1. Let $\GG$ be the associated reductive group over~$\QQ$, and denote by $[\mu]$ the associated conjugacy class of cocharacters of~$\GG$. We assume that $\GG$ is connected, i.e., we exclude the case (D) of \cite{ViWd}~Remark~1.1.

Let $\bar \QQ$ be the algebraic closure of $\QQ$ in $\CC$. Then $[\mu]$ is already defined over $\bar \QQ$. Let $E$ be the reflex field associated with $\Dscr$, i.e.~the field of definition of~$[\mu]$. It is a finite extension of~$\QQ$ contained in $\bar \QQ$. Let $\bar \QQ_p$ be an algebraic closure and fix an embedding $\bar \QQ \mono \bar \QQ_p$. This determines a place $v$ of $E$ over $p$. Let $E_v \subseteq \QQbar_p$ be the $v$-adic completion of $E$, $O_{E_v}$ its ring of integers, and let $\kappa$ be its residue field. The assumption on $\Dscr$ to be unramified implies that $E_v$ is an unramified extension of $\QQ_p$. Let $\kgbar$ be the residue field of the ring of integers of $\QQbar_p$. This is an algebraic closure of $\kappa$.

Let $\AA^p_f$ be the ring of finite adeles over $\QQ$ with trivial $p$-th component and let $C^p \subset \GG(\AA^p_f)$ be a compact open subgroup. We denote by $\Ascr = \Ascr_{\Dscr,C^p}$ the moduli space defined by Kottwitz~\cite{Ko_ShFin} \S5 (see also \cite{ViWd}~Section~1).

Then $\Ascr$ is an algebraic Deligne-Mumford stack which is smooth over $O_{E_v}$. If $C^p$ is sufficiently small, $\Ascr$ is representable by a smooth quasi-projective scheme over $O_{E_v}$ (see loc.~cit.\ or \cite{Lan}~1.4.1.11 and~1.4.1.13). We denote its special fiber by
\[
\Ascr_0 = \Ascr_{\Dscr,C^p,0} := \Ascr_{\Dscr,C^p} \otimes_{O_{E_v}} \kappa.
\]

\subsection*{Group theoretical notation}
Let $\Gtilde$ be the $\ZZ_p$-group scheme of $O_{\BB}$-linear symplectic similitudes of $\Lambda$. This is a reductive group scheme over $\ZZ_p$ whose generic fiber is $\GG_{\QQ_p}$. We denote by $G$ its special fiber. This is a connected reductive group over $\FF_p$ and hence quasi-split (as any reductive group over a finite field). We fix a maximal torus $T$ of $G$ and a Borel subgroup $B$ of $G$ containing $T$ (both defined over $\FF_p$). For every $\FF_p$-algebra $R$ we set
\[
G_R := G \otimes_{\FF_p} R, \qquad T_R := T \otimes_{\FF_p} R, \qquad B_R := B \otimes_{\FF_p} R
\]
We denote by $X^*(T)$ (resp.~$X_*(T)$) the group of characters (resp.~of cocharacters) of $T_{\kgbar}$. Let $(W,I)$ be the Weyl group together with its set of simple reflections of $(G,B,T)$.

We denote by $w_0 \in W$ the element of maximal length. If $\Phi \subset X^*(T)$ is the set of roots of $(G,T)$ and $\Delta \subset \Phi$ the root basis corresponding to $B$, then there exists a unique involution $\iota$ of $\Phi$ such that $\iota(\alpha) = -w_0(\alpha)$ for $\alpha \in \Phi$. It preserves $\Delta$. The induced involution of the Coxeter system $(W,I)$, given by $s_{\alpha} \to s_{\iota(\alpha)}$, where $s_{\alpha} \in W$ is the reflection corresponding to $\alpha \in \Phi$, is the conjugation with $w_0$. If $S \subseteq W$ is a subset, we denote by $S^{\rm opp} := \leftexp{w_0}{S}$ the image of $S$ under this involution.

The Galois group $\Gamma := \Gal(\kgbar/\FF_p)$ acts on $X^*(T)$, on $X_*(T)$, and on the Coxeter system $(W,I)$. It is topologically generated by the geometric Frobenius automorphism $\sigma \in \Gamma$ (i.e. $\sigma^{-1}$ is the automorphism $x \sends x^p$ of $\kgbar$). The automorphism of the Coxeter systems $(W,I)$ induced by $\sigma$ is denoted by $\bar\varphi$. Note that $\bar\varphi$ is also the automorphism induced by the geometric Frobenius $\varphi\colon G_{\FF_p} \to G_{\FF_p}$.

For any subsets $J,K \subseteq I$, we denote by $W_J$ the subgroup of $W$ generated by $J$ and by $\leftexp{J}{W}$ (resp.\ $W^{K}$, resp.\ $\doubleexp{J}{W}{K}$) the set of $w \in W$ that are of minimal length in the left coset $W_Jw$ (resp.\ in the right coset $wW_{K}$, resp.\ in the double coset $W_JwW_{K}$). Then $\doubleexp{J}{W}{K} = \leftexp{J}{W} \cap W^K$.

The Coxeter group is endowed with the Bruhat order which we denote ``$\leq$'': For $x, w\in W$ we have $x\leq w$ if and only if for some (or, equivalently, any) reduced expression $w=s_1\cdots s_n$ as a product of simple reflections $s_i\in I$, one gets a reduced expression for $x$ by removing certain $s_i$ from this product. The minimal number $\ell$ needed to express $w \in W$ as product of $\ell$ simple reflections is the \emph{length of $w$} and denoted $\ell(w)$.

For $J \subseteq I$ let $\kappa(J)$ be its field of definition, i.e. $\kappa(J)$ is the finite extension of $\FF_p$ in $\kgbar$ such that
\begin{equation}\label{DefGammaJ}
\Gamma_J := \Gal(\kgbar/\kappa(J)) = \set{\gamma \in \Gamma}{\gamma(J) = J}
\end{equation}
Let $P_J \subseteq G_{\kappa(J)}$ be the unique parabolic subgroup of type $J$ containing $B$, and let $\Par_J = G_{\kappa(J)}/P_J$ be the projective $\kappa(J)$-scheme of parabolics of type $J$ of $G$.

Let $\Ttilde$ be a maximal torus of $\Gtilde$ such that $\Ttilde \otimes_{\ZZ_p} \FF_p = T$ (such a $\Ttilde$ always exists because the scheme of maximal tori of $\Gtilde$ is smooth, \cite{SGA3}~Exp.~XV, 8.15) and let $\Tbf$ be its generic fiber. Let $\QQ^{\rm nr}_p$ denote the maximal unramified extension of $\QQ_p$ in $\QQbar_p$ and identify $\Gal(\QQ^{\rm nr}_p/\QQ_p)$ with $\Gamma$. Then the cocharacter group $X_*(\Tbf \otimes_{\QQ_p} \QQ^{\rm nr}_p)$ is isomorphic to $X_*(T)$ as $\Gamma$-modules. Every element in the conjugacy class $[\mu]$ is conjugate via some element in $\Gtilde(\QQ^{\rm nr}_p)$ to an element of $X_*(\Tbf \otimes_{\QQ_p} \QQ^{\rm nr}_p)$ which is unique up to conjugation with an element of the normalizer of $\Tbf \otimes_{\QQ_p} \QQ^{\rm nr}_p$. Thus via the identification $X_*(\Tbf \otimes_{\QQ_p} \QQ^{\rm nr}_p) = X_*(T)$ we may consider the conjugacy class $[\mu]$ as an element of $X_*(T)/W$.

Let $\mu$ be the $B$-dominant representative in $X_*(T)$ in the conjugacy class $[\mu]$. As $G$ is quasi-split, it is defined over the field of definition of its conjugacy class, i.e. over $\kappa$. We set
\begin{equation}\label{DefineJ}
J := \set{i \in I}{\langle \mu, \alpha_i\rangle = 0},
\end{equation}
where $\alpha_i \in X^*(T)$ is the simple root corresponding to $i \in I$. We also set
\begin{equation}\label{DefineK}
K := \leftexp{w_0}{\bar\varphi(J)} = \bar\varphi(J)^{\rm opp}.
\end{equation}
As $\kappa$ is the field of definition of $[\mu]$, one has for the fields of definitions of $J$ and $K$
\[
\kappa(J) = \kappa(K) = \kappa.
\]


\subsection*{The Newton stratification and the $\mu$-ordinary locus}
As in~\cite{ViWd}~Section~7.2 we denote by
\[
{\rm Nt}\colon \Ascr_0 \to B(\GG_{\QQ_p},\mu)
\]
the map that sends each point $s$ of $\Ascr_0$ to the isogeny class of the $p$-divisible group with $\Dscr$-structure given by a geometric point lying over $s$. For $b \in B(\GG_{\QQ_p},\mu)$ we denote by $\Ncal_b := {\rm Nt}^{-1}(b)$ the corresponding Newton stratum.

The set $B(\GG_{\QQ_p},\mu)$ is finite and partially ordered. It contains a unique maximal element $b_\mu$ and the corresponding Newton stratum $\Ncal_\mu := \Ncal_{b_\mu}$ is called the \emph{$\mu$-ordinary Newton stratum}. It is open and dense in $\Ascr_0$ by the main result of~\cite{Wd_Ord}.


\section{Definition of the Bruhat stratification}\label{BRUHAT}

In \cite{ViWd}~2.1 Viehmann and the author constructed a morphism $\zeta$ from $\Ascr_0$ to the stack of so-called $\Dscr$-zips which we identified with the stack of $G$-zips of type $\mu$ (in the sense of \cite{PWZ2}) in \cite{ViWd}~4. The corresponding zip stratification (\cite{PWZ2}~(3.27) or \cite{Wd_Bruhat}~Definition~2.8) is nothing but the Ekedahl-Oort stratification. As explained in \cite{Wd_Bruhat} we also obtain a Bruhat stratification by composing $\zeta$ with the canonical morphism from the stack of $G$-zips of type $\mu$ into the Bruhat stack
\begin{equation}\label{DefBruhatStack}
\Bcal_{J,K} := [G_{\kappa}\bs (\Par_J \times \Par_K)]
\end{equation}
defined in \cite{Wd_Bruhat}. Here $[G \bs X]$ denotes the quotient stack over a scheme $S$ if a group scheme $G$ over $S$ acts from the left on a scheme $X$ over $S$. We denote this composition by
\[
a\colon \Ascr_0 \to \Bcal_{J,K}.
\]
Below we will give a more direct description of the morphism $a$.

We recall the properties of $\Bcal_{J,K}$ shown in \cite{Wd_Bruhat} (note that the notation here is considerably simplier because $G$ is by hypothesis connected).
\begin{assertionlist}
\item
$\Bcal_{J,K}$ is a smooth algebraic stack of finite type over $\kappa$. It is of dimension $\dim(G) - \dim(P_J) - \dim(P_K)$ over $\kappa$. Note that with $K$ defined as in \eqref{DefineK} we have $\dim P_J = \dim P_K$ and hence $\dim \Bcal_{J,K} = - \dim(L_J)$, where $L_J$ is a Levi subgroup of $P_J$.
\item
To describe its underlying topological space recall that one may endow any partially ordered set $X$ with a topology by defining the open sets as those subsets $U$ of $X$ such that for all $u \in U$ and $x \in X$ with $x \geq u$ one has $x \in U$. This defines a fully faithful functor from the category of partially open sets (morphisms are order preserving maps) to the category of topological spaces.

Then the underlying topological space of $\Bcal_{J,K}$ can be identified with the space attached to the partially ordered set $\Gamma_J\bs \doubleexp{J}{W}{K}$. Here $\doubleexp{J}{W}{K}$ carries the Bruhat order which is preserved by the action by $\Gamma_J$, and $\Gamma_J\bs \doubleexp{J}{W}{K}$ is endowed with the induced order.

Over the algebraic closure, the underlying topological space of $\Bcal_{J,K} \otimes_{\kappa} \kgbar$ is given by the partially ordered set $\doubleexp{J}{W}{K}$ and the projection $\Bcal_{J,K} \otimes_{\kappa} \kgbar \to \Bcal_{J,K}$ induces the canonical map $\doubleexp{J}{W}{K} \to \Gamma_J\bs \doubleexp{J}{W}{K}$ on underlying topological spaces.
\item 
For each $x \in \doubleexp{J}{W}{K}$ consider the reduced locally closed substack $\Gscr_x$ of $\Bcal_{J,K} \otimes \kgbar$ consisting only of the point $x$. It is already defined over $\kappa(x)$ and it is smooth of dimension $\dim(P_J x P_K) - \dim P_K - \dim P_J = \ell_{J,K}(x) - \dim P_K$, where
\[
\ell_{J,K}(x) := \ell(x) + \ell(w_{0,K}) - \ell(w_{0,J_w}).
\]
Here $w_{0,K}$ denotes the element of maximal length in $W_K$ and $J_w := K \cap w^{-1}Jw$.
\end{assertionlist}

\subsection*{$\Dscr$-module structure on the De Rham homology}
Let $S$ be a $\kappa$-scheme and let $(A,\iota,\lambda,\eta)$ be an $S$-valued point of~$\Ascr_0$. Let $\DD(A[p])$ be the \emph{covariant} Dieudonn\'e crystal of the $p$-torsion and let $M(A)$ be its evaluation in the trivial PD-thickening $(S,S,0)$. Then there is an $\Oscr_S$-linear functorial isomorphism $M(A) \cong H_1^{\DR}(A/S)$, where the De Rham homology is defined as the $\Oscr_S$-linear dual of the De Rham cohomology $H^1_{\DR}(A/S)$ (\cite{BBM_DieudII}~Prop.~2.5.8). Thus $M(A)$ is a locally free $\Oscr_S$-module of rank $\dim_{\QQ}(V)$. Via functoriality it is endowed with an $O_\BB$-action (induced by $\iota$) and with a similitude class of perfect alternating forms (induced by $\lambda$), i.e. $M(A)$ is equipped with the structure of a $\Dscr$-module in the following sense.

\begin{definition}\label{DefDModule}
Let $T$ be a $\ZZ_p$-scheme. A \emph{$\Dscr$-module over $T$} is a locally free $\Oscr_T$-module $M$ of rank $\dim_{\QQ}(V)$ endowed with an $O_\BB$-action and the similitude class of a symplectic form $\lrangle$ such that $\langle bm,m' \rangle = \langle m,b^*m'\rangle$ for all $b \in O_\BB$ and local sections $m,m'$ of $M$.

An \emph{isomorphism $M_1 \iso M_2$ of $\Dscr$-modules} over $T$ is an $\Oscr_T \otimes O_\BB$-linear symplectic similitude.
\end{definition}

Here we call two perfect pairings $\lrangle_1$ and $\lrangle_2$ on a finite locally free $\Oscr_T$-module $M$ \emph{similar} if there exists an open affine covering $T = \bigcup_j V_j$ and for all $j$ a unit $c_j \in \Gamma(V_j,\Oscr_{V_j}^{\times})$ such that $\langle m,m' \rangle_2 = c_j\langle m,m' \rangle_1$ for all $m,m' \in \Gamma(V_j,M)$.

For a morphism $f\colon T' \to T$ of $\ZZ_p$-schemes and a $\Dscr$-module $M$ there is the obvious notion of a pull back $f^*M = M_{T'}$ of $M$ to a $\Dscr$-module on $T'$. Altogether we obtain the category $\widetilde{\DMod}$ of $\Dscr$-modules fibered over the category of $\ZZ_p$-schemes. As descent for finite locally free modules is effective for fpqc-morphisms, $\widetilde{\DMod}$ is a stack for the fpqc topology. In fact it is the classifying stack of $\Gtilde$:

\begin{proposition}\label{StackDMod}
$\widetilde{\DMod} = [\Gtilde \backslash \Spec(\ZZ_p)]$.
\end{proposition}

Here we endow $\Spec(\ZZ_p)$ with the trivial $\Gtilde$-action.

\begin{proof}
The $\ZZ_p$-module $\Lambda$ together with the induced $O_\BB$-action and the induced pairing is a $\Dscr$-module over $\ZZ_p$. Its automorphism group scheme is $\Gtilde$. Moreover, \cite{RZ_Period}~Theorem~3.16 shows that all $\Dscr$-modules are \'etale locally isomorphic to the $\Dscr$-module $\Lambda_T$. This shows the claim.
\end{proof}

We set
\[
\DMod := \widetilde{\DMod} \otimes_{\ZZ_p} \FF_p = [G \backslash \Spec(\FF_p)].
\]
We will also use the following torsor $\Ascr^{\#}$ over $\Ascr$. For every $O_{E_v}$-scheme~$S$ the $S$-valued points of $\Ascr^{\#}$ are given by quintuples $(A,\iota,\lambda,\eta,\alpha)$, where $(A,\iota,\lambda,\eta) \in \Ascr(S)$ and where $\alpha$ is an $O_\BB$-linear symplectic similitude\break $H_1^{\DR}(A/S) \iso \Lambda \otimes_{\ZZ_p} \Oscr_S$.

By Proposition~\ref{StackDMod}, $\Ascr^{\#}$ is a $\Gtilde_{O_{E_v}}$-torsor over~$\Ascr$ for the \'etale topology. We also set $\Ascr^{\#}_0 := \Ascr^{\#} \otimes_{O_{E,v}} \kappa$.

\subsubsection*{Hodge filtration and conjugate filtration on the De Rham homology}
Now let $(A,\iota,\lambda,\eta,\alpha)$ be an $S$-valued point of $\Ascr^{\#}_0$. Then the first de Rham homology $M(A) = H_1^{\DR}(A/S)$ (i.e., the $\Oscr_S$-linear dual of the first hypercohomology of the relative de Rham complex $\Omega^{\bullet}_{A/S}$) comes with two totally isotropic locally direct summands, namely
\begin{equation}\label{HodgeFil}
C := \omega_{A\vdual} := f\vdual_*\Omega^1_{A\vdual/S}, \qquad D := R^1f_*(\Hscr^0(\Omega^{\bullet}_{A/S}))^{\perp}.
\end{equation}
Here $f\colon A \to S$ (resp.~$f\vdual\colon A\vdual \to S$) denotes the structure morphism of $A$ (resp.~of the dual abelian scheme $A\vdual$).

We define $P$ (resp.~$Q$) to be the stabilizer of $\alpha(C)$ (resp.~$\alpha(D)$) in $G_S$. Then $P$ (resp.~$Q$) is a parabolic subgroup of type $J$ (resp.~$K = \bar\varphi(J)^{\rm opp}$) of $G_S$ by \cite{ViWd}~Prop.~1.13.

As the formation of $C$ and $D$ is functorial, we obtain a morphism
\[
a^{\#}\colon \Ascr_0^{\#} \to \Par_J \times \Par_K.
\]
We endow $\Par_J \times \Par_K$ with the obvious diagonal left action of $G_{\kappa}$ and denote by
\[
\Bcal_{J,K} := \Bcal_{J,K}(G) := [G_{\kappa}\bs (\Par_J \times \Par_K)]
\]
the quotient stack. Then $a^{\#}$ is $G_{\kappa}$-equivariant and hence it induces the morphism
\begin{equation}\label{EqAMorphism}
a\colon \Ascr_0 \to \Bcal_{J,K}.
\end{equation}
Applying \cite{Wd_Bruhat}~Definition~1.13, the morphism $a$ yields the \emph{Bruhat stratification of $\Ascr_0$}, indexed by $\Gamma_J\backslash\doubleexp{J}{W}{K}$, i.e. we obtain a decomposition
\begin{equation}\label{EqBruhat}
\Ascr_0 = \bigcup_{[x] \in \Gamma_J\backslash\doubleexp{J}{W}{K}}\AStrat{[x]}
\end{equation}
of $\Ascr_0$ into locally closed substacks.

\begin{definition}
The locally closed substacks $\AStrat{[x]}$ for $[x] \in \Gamma_J\backslash\doubleexp{K}{W}{J}$ are called the \emph{Bruhat strata of $\Ascr_0$}.
\end{definition}

We also define geometric Bruhat strata as follows. Let $\kappa \subseteq k \subset \kgbar$ be the splitting field of the reductive group $G$. Then $\Gal(\kgbar/k)$ acts trivially on $(W,I)$ and the underlying topological space of $\Bcal_{J,K} \otimes_{\kappa} k$ is $\doubleexp{K}{W}{J}$. The residue gerbe $\Gscr_x$ attached to $x \in \doubleexp{K}{W}{J}$ is a locally closed substack of $\Bcal_{J,K} \otimes_{\kappa} \kappa(x)$, where $\kappa \subseteq \kappa(x) \subseteq k$ is the field of definition of $x$. We set
\begin{equation}\label{DefGeomBruhat}
\AStrat{x} := a^{-1}(\Gscr_{x} \otimes_{\kappa(x)} k) \subseteq \Ascr_0 \otimes_{\kappa} k.
\end{equation}
This is a locally closed substack of $\Ascr_0 \otimes_\kappa k$, called a \emph{geometric Bruhat stratum}. Then one has for $[x] = \Gamma_J\cdot x \in \Gamma_J\backslash\doubleexp{J}{W}{K}$ a decomposition into open and closed substacks
\[
\AStrat{[x]} \otimes_{\kappa} k = \coprod_{z \in \Gamma_J\cdot x}\AStrat{z}.
\]
Altogether we obtain a disjoint decomposition into locally closed substacks
\[
\Ascr_0 \otimes_{\kappa} k = \bigcup_{x \in \doubleexp{K}{W}{J}} \AStrat{x},
\]
which we call the \emph{geometric Bruhat stratification of $\Ascr_0$}.

\begin{example}\label{Siegel}
Consider the Siegel case, i.e., we choose $\BB = \QQ$ in the Shimura-PEL-datum $\Dscr$, and hence $\Gtilde = \GSp(\Lambda,\lrangle)$. Set $g := \dim_\QQ(V)/2$. Let us explain that in this case the Bruhat stratification is nothing but a scheme-theoretic version of the stratification by the $a$-number defined by Oort in~\cite{Oort3}.

The reflex field $E$ is equal to $\QQ$ and hence $\kappa = \FF_p$. Then a $\Dscr$-module over an $\ZZ_p$-scheme $S$ is a locally free $\Oscr_S$-module of rank $2g$ endowed with the similitude class of a symplectic pairing. We have $J = K$ and attaching to a Lagrangian $C$ in the free $\Dscr$-module $\Lambda_S := \Lambda \otimes_{\ZZ_p} \Oscr_S$ (i.e., $C$ is a totally isotropic locally direct summand of rank $g$) its stabilizer in $G_S$ yields a bijection between the set of Lagrangians in $\Lambda_S$ and set of parabolic subgroups of type $J$ of $G_S$.

Let $k$ be a field extension of $\kappa$, $C_i, D_i \subset \Lambda_k$ Lagrangians ($i = 1,2$) and let $P_i = \Stab_{G_k}(C_i)$ and $Q_i = \Stab_{G_k}(D_i)$ be the correpsonding subgroups of type $J$. Then two pairs $(P_1,Q_1)$ and $(P_2,Q_2)$ are in the same $G(k)$-orbit if and only if $\dim_k(C_1 \cap D_1) = \dim_k(C_2 \cap D_2) \in \{0,\dots,g\}$. This shows that in this case we can identify $\doubleexp{J}{W}{J}$ with $\{0,\dots,g\}$ via a map $w \sends n(w)$. Moreover we have $w \leq w'$ (w.r.t. the Bruhat order) if and only if $n(w) \geq n(w')$.

Now assume that $k$ is perfect and let $(A,\lambda,\eta,\alpha)$ be a $k$-valued point of $\Ascr^{\#}_0$. Denote by $\sigma$ the absolute Frobenius on $k$. Thus $\alpha$ is a symplectic similitude $M(A) \cong \Lambda_k := k \otimes_{\ZZ_p} \Lambda$, where the symplectic form on $M(A)$ is induced by $\lambda$ and on $\Lambda_k$ by $\lrangle$. Set $C := \alpha(\omega_{A\vdual})$ and $D := \alpha(R^1f_*(\Hscr^0(\Omega^{\bullet}_{A/S}))^{\perp})$. These are totally isotropic $g$-dimensional subspaces of $\Lambda_k$. Recall that for any $k$-vector space $W$ we set $W^{(p)} = k \otimes_{\sigma,k} W$. Let $F\colon M(A)^{(p)} \to M(A)$, $V\colon M(A) \to M(A)^{(p)}$ be the $k$-linear maps induced by Verschiebung and Frobenius, respectively. By transport of structure via $\alpha$ we obtain $k$-linear maps $F\colon \Lambda_k^{(p)} \to \Lambda_k$ and $V\colon \Lambda_k \to \Lambda_k^{(p)}$ such that $C^{(p)} = V(\Lambda_k)$ and $D = F(\Lambda^{(p)}_k)$ (this is the dual version of \cite{Oda}~5.11, where De Rham cohomology and contravariant Dieudonn\'e theory are considered).

Let $F^{\flat}\colon \Lambda_k \to \Lambda_k$ be the $\sigma$-linear map corresponding to $F$, i.e. $F^{\flat}(x) = F(1 \otimes x)$ for $x \in \Lambda_k$. As $\sigma$ is bijective, the image of $F$ and of $F^{\flat}$ coincide.

We also denote by $V^{\flat}\colon \Lambda_k \to \Lambda_k$ the $\sigma^{-1}$-linear map corresponding to $V$, i.e., $V^{\flat} = \tau \circ V$ with $\tau\colon \Lambda_k^{(p)} \to \Lambda_k$, $(a \otimes x) = \sigma^{-1}(a)x$. On the other hand, as $\Lambda_k$ has an $\FF_p$-rational stucture $\Lambda \otimes_{\ZZ_p} \FF_p$, there are a canonical $k$-linear isomorphism $\gamma\colon \Lambda_k^{(p)} \iso \Lambda_k$ and a $\sigma$-linear map $\sigma_\Lambda\colon \Lambda_k \to \Lambda_k$, $\alpha \otimes z \sends \alpha^p \otimes z$ for $\alpha \in k$ and $z \in \Lambda$. Finally let $\beta\colon \Lambda_k \to \Lambda_k^{(p)}$ the $\sigma$-linear map $x \sends 1 \otimes x$ for $x \in \Lambda_k$. Then
\begin{align}
\sigma_\Lambda \circ \tau &= \gamma,\label{FrobStuff1}\\
\gamma \circ \beta &= \sigma_\Lambda.\label{FrobStuff2}
\end{align}
Thus we find
\[
\sigma_\Lambda(C) \eqann{\eqref{FrobStuff2}} \gamma(\beta(C)) = \gamma(C^{(p)}) = \gamma(V(\Lambda_k)) \eqann{\eqref{FrobStuff1}} \sigma_\Lambda(V^{\flat}(\Lambda_k))
\]
and hence $V^{\flat}(\Lambda_k) = C$.

Therefore $(A,\lambda,\eta)$ is in the Bruhat stratum corresponding to $w \in \doubleexp{J}{W}{J}$ if and only if $\dim_k (F^{\flat}M(A) \cap V^{\flat}M(A)) = n(w)$, i.e. if and only if the point $(A,\lambda,\eta)$ has $a$-number equal to $n(w)$ in the sense of~\cite{Oort3}.
\end{example}

The morphism $a$ induces a continuous map $\Ascr_0(\kgbar) \to \doubleexp{J}{W}{K}$. This can also be expressed as the following semi-continuity result. 

\begin{remark}
As usual endow $\doubleexp{J}{W}{K}$ with the Bruhat order. For all $x \in \doubleexp{J}{W}{K}$ the set
\[
\set{s \in \Ascr_0(\kgbar)}{a(s) \leq x}
\]
is closed. Below (Corollary~\ref{AStratProp}) we will show that this set is the closure of $\set{s \in \Ascr_0(\kgbar)}{a(s) = x}$.

For instance we can apply this in the Siegel case (Example~\ref{Siegel}). There we identified the partially ordered sets $(\doubleexp{J}{W}{K},\leq)$ and $(\{0,\dots,g\},\geq)$. Hence for $d \in \{0,\dots,g\}$ the locus where the $a$-number $\geq d$ is always a closed subset. Moreover, it is the closure of the locus of points where the $a$-number is equal to $d$.
\end{remark}


\section{Ekedahl-Oort strata and Bruhat strata}\label{EOAndAStrat}
The Ekedahl-Oort stratification defined in~\cite{ViWd} is a special case of the zip stratification introduced in \cite{PWZ2}. It is a refinement of the Bruhat stratification and the results of \cite{Wd_Bruhat} allow to give a simple description of the Ekedahl-Oort strata that are contained in a given Bruhat stratum.

For this recall that the Ekedahl-Oort stratification is a decomposition into locally closed substacks
\begin{equation}\label{EqEO}
\Ascr_0 = \bigcup_{[w] \in \Gamma_J\bs\leftexp{J}{W}} \Ascr^{[w]}.
\end{equation}
As above we also have geometric Ekedahl-Oort strata $\Ascr^w$ for $w \in \leftexp{J}{W}$ defined over a splitting field $k$ of $G$. By \cite{ViWd}~Section~6 and Section~9, the $\Ascr^w$ are quasi-affine, smooth of dimension $\ell(w)$, and the closure of an Ekedahl-Oort stratum (resp.~a geometric an Ekedahl-Oort stratum) is the union of Ekedahl-Oort strata (resp.~of geometric Ekedahl-Oort stratum). It suffices to compare geometric Ekedahl-Oort strata and geometric Bruhat strata.

\begin{proposition}\label{EOaStrata}
For $x \in \doubleexp{J}{W}{K}$ one has
\[
\AStrat{x} = \bigcup_{y \in \leftexp{K \cap x^{-1}Jx}{W_K}}\Ascr^{xy}
\]
\end{proposition}

\begin{proof}
In \cite{Wd_Bruhat}~Proposition~2.10 it is shown that the canonical morphism from the stack of $G$-zips of type $\mu$ to the Bruhat stack $\Bcal_{J,K}$ is given on the underlying topological spaces by attaching to $w \in \leftexp{J}{W}$ the unique element of minimal length in $wW_K$. This is a surjective map
\[
\pi\colon \leftexp{J}{W} \to \doubleexp{J}{W}{K}.
\]
Hence for $x \in \doubleexp{J}{W}{K}$ one has
\begin{equation}
\AStrat{x} = \bigcup_{w \in \pi^{-1}(x)}\Ascr^w,
\end{equation}
But by a corollary of a result of Howlett on Coxeter groups (e.g., \cite{PWZ1} Prop.~2.8) one has
\[
\pi^{-1}(x) = \set{xy}{y \in \leftexp{K \cap x^{-1}Jx}{W_K}}.\qedhere
\]
\end{proof}

\begin{corollary}\label{SingleEO}
For $x \in \doubleexp{J}{W}{K}$ the corresponding Bruhat stratum $\AStrat{x}$ consists of a single Ekedahl-Oort stratum if and only if $K = x^{-1}Jx$.
\end{corollary}

\begin{proof}
By Proposition~\ref{EOaStrata}, $\AStrat{x}$ consists of a single Ekedahl-Oort stratum if and only $K \cap x^{-1}Jx = K$. As $J$ and $K$ have the same number of elements, this condition is equivalent to $K = x^{-1}Jx$.
\end{proof}

\begin{example}
Assume that $J = K$ (e.g., if $G$ is a group of Dynkin type $C$, because then the element $w_0$ of maximal length is central and $\bar\varphi(J) = J$). Let $1 \in \doubleexp{K}{W}{J}$ be the element of minimal length. Then $\pi^{-1}(1) = \{1\}$ and $\AStrat{1}$ is the superspecial Ekedahl-Oort stratum (in the sense of \cite{ViWd}~Example~4.16).

For $J \ne K$, $\AStrat{1}$ is the union of more than one Ekedahl-Oort stratum.
\end{example}


\subsection*{The maximal Bruhat stratum}

Let $\xtilde \in \doubleexp{J}{W}{K}$ be the element of maximal length. The corresponding geometric Bruhat stratum $\AStrat{\xtilde}$ is called the \emph{maximal Bruhat stratum}.

\begin{remark}\label{MaxGeom}
As the Galois action by $\Gamma_J$ on $\doubleexp{J}{W}{K}$ preserves the length of elements, we have $\AStrat{\xtilde} = \AStrat{[\xtilde]}$, i.e., the maximal Bruhat stratum is already defined over $\kappa$.
\end{remark}

The maximal Bruhat stratum is open and dense in $\Ascr_0$: If we consider $\doubleexp{J}{W}{K}$ as a topological space, the maximality of $\xtilde$ means that $\{\xtilde\}$ is open in $\doubleexp{J}{W}{K}$. Thus $\AStrat{\xtilde}$ is an open in $\Ascr_0$. Moreover it contains by Proposition~\ref{EOaStrata} the $\mu$-ordinary Ekedahl-Oort stratum, i.e. the Ekedahl-Oort stratum corresponding to the maximal element $w_\mu$ in $\leftexp{J}{W}$ (\cite{ViWd}~Example~4.16). By~\cite{ViWd}~Theorem~6.1 the $\mu$-ordinary Ekedahl-Oort stratum is dense in $\Ascr_0$ and thus $\AStrat{\xtilde}$ is dense.
The density of $\AStrat{\xtilde}$ will also follow from Corollary~\ref{AStratProp} below.

In general, $\AStrat{\xtilde}$ is not equal to $\mu$-ordinary Ekedahl-Oort stratum (which is also the $\mu$-ordinary Newton stratum by the first main result of~\cite{Mo_SerreTate}; see also \cite{Wortmann} for a purely group-theoretical proof).

\begin{theorem}\label{MuOrdAStrat}
The following assertions are equivalent.
\begin{equivlist}
\item
The generic Bruhat stratum $\AStrat{\xtilde}$ is equal to the $\mu$-ordinary Newton stratum.
\item
$\bar\varphi(J) = J$.
\item
$E_v = \QQ_p$.
\item
The ordinary locus of $\Ascr_0$ (i.e. the locus of points $\APlus$ where the underlying abelian scheme $A$ is ordinary) is non-empty.
\item
The ordinary locus of $\Ascr_0$ is equal to the $\mu$-ordinary locus.
\end{equivlist}
\end{theorem}

\begin{proof}
Assertion~(ii) means that $\kappa(J) = \FF_p$ and thus is equivalent to~(iii). The equivalence of~(iii),~(iv) and~(v) is \cite{Wd_Ord}~(1.6.3).  Thus it remains to show that~(ii) is equivalent to the equality of $\AStrat{\xtilde}$ and the $\mu$-ordinary Ekedahl-Oort stratum (by the above mentioned result of Moonen). By Corollary~\ref{SingleEO} this equality holds if and only if $K = \xtilde^{-1}J\xtilde = J^{\rm opp}$. But by definition $K = \bar\varphi(J)^{\rm opp}$. This shows the equivalence of~(i) and~(ii).
\end{proof}

\begin{example}
The equivalent conditions of Proposition~\ref{MuOrdAStrat} are satisfied in the following cases.
\begin{assertionlist}
\item
All connected components of the Dynkin diagram of $G$ are of Dynkin type C.
\item
$G$ is split.
\end{assertionlist}
\end{example}


\section{Properties of the Bruhat strata}\label{PROP}

\begin{theorem}\label{ASmooth}
The morphism $a$~\eqref{EqAMorphism} is smooth and surjective.
\end{theorem}

Before giving the proof of this theorem we deduce some properties of the geometric Bruhat strata. For this we first introduce the following notation. For $x \in \doubleexp{J}{W}{K}$ let $x^{J,K}$ be the element of maximal length in $\leftexp{J}{W} \cap W_JxW_K$. Then $x^{J,K} = xx_{J,K}$, where $x_{J,K}$ is the element of maximal length in $\leftexp{K \cap x^{-1}Jx}{W_K}$ and we have
\begin{equation}\label{DefEllJK}
\ell(x^{J,K}) = \ell_{J,K}(x) = \ell(x) + \ell(x_{J,K})
\end{equation}
by a result of Howlett (\cite{PWZ1}~2.7 and~2.8).

\begin{corollary}\label{AStratProp}
Let $x \in \doubleexp{J}{W}{K}$.
\begin{assertionlist}
\item\label{AStratProp1}
The corresponding geometric Bruhat stratum $\AStrat{x}$ is smooth of pure dimension $\ell(x^{J,K})$~\eqref{DefEllJK}. In particular, all Bruhat strata are non-empty.
\item\label{AStratProp2}
The closure of $\AStrat{x}$ is given by
\[
\overline{\AStrat{x}} = \bigcup_{x' \leq x}\AStrat{x'},
\]
where $\leq$ denotes the Bruhat order on $\doubleexp{J}{W}{K}$.
\end{assertionlist}
\end{corollary}

\begin{proof}
As the morphism $a$ is surjective, all Bruhat strata are non-empty. Let $\Gscr_x \subset \Bcal_{J,K} \otimes \kappa(x)$ be the residue gerbe of $x$. Then the cartesian diagram
\[\xymatrix{
\AStrat{x} \ar[r] \ar[d]^{\leftexp{x}{a}} & \Ascr_0 \otimes \kappa(x) \ar[d]^{a_{\kappa(x)}} \\
\Gscr_x \ar[r] & \Bcal_{J,K} \otimes \kappa(x)
}\]
shows that the smoothness of $a$ implies the smoothness of $\leftexp{x}{a}$. As $\Gscr_x$ is smooth over $\kappa(x)$, $\AStrat{x}$ is smooth over $\kappa(x)$.

Finally, as $a$ is smooth, it preserves codimension. By \cite{Wd_Bruhat}~Proposition~1.12, $\Gscr_x$ has in $\Bcal_{J,K} \otimes \kappa(x)$ pure codimension $\dim G - \dim P_J - \ell(x^{J,K})$. As $\dim\Ascr_0 = \dim G - \dim P_J$, we obtain that $\AStrat{x}$ is of pure dimension $\ell(x^{J,K})$.

As $a$ is smooth, it is an open morphism. Hence~(2) follows from the description of the underlying topological space of $\Bcal_{J,K}$ (\cite{Wd_Bruhat}~Proposition~1.10) and the fact that for open morphisms taking closures commutes with taking inverse images. 
\end{proof}

\begin{proof}[Proof of Theorem~\ref{ASmooth}]
The morphism $a$ is the composition of the morphism $\zeta$, which attaches to every point of $\Ascr_0$ its $G$-zip of type $\mu$, and the morphism from the stack of $G$-zips of type $\mu$ to the Bruhat stack $\Bcal_{J,K}$. This morphism is surjective by Theorem~9.1 of~\cite{ViWd}. As explained in the proof of Proposition~\ref{EOaStrata}, the morphism from the stack of $G$-zips of type $\mu$ to the Bruhat stack $\Bcal_{J,K}$ is also surjective. Therefore $a$ is surjective.

Moreover, source and target of $a$ are both smooth over $\Spec \kappa$. Therefore it suffices to show that $a$ is surjective on tangent spaces. By this we mean the following. Let $k$ be an algebraically closed extension of $\kappa$, let $k[\eps]$ be the ring of dual numbers over $k$, and let
\begin{equation}\label{EqTangent}
\begin{aligned}\xymatrix{
\Spec k \ar[r]^x \ar[d] & \Ascr_0 \ar[d]^a \\
\Spec k[\eps] \ar[r]^{\xtilde} & \Bcal_{J,K}
}\end{aligned}
\end{equation}
be a 2-commutative diagram. Then we have to show that there exists a morphism $\Spec k[\eps] \to \Ascr_0$ which 2-commutes with~\eqref{EqTangent}.

For this we use Dieudonn\'e displays defined by Zink and Lau (\cite{Zink_Dieudonne}, \cite{Lau_Duality}, \cite{Lau_DieudonneDisplays}; see also \cite{ViWd}~3.1 for a short reminder). We denote the Zink rings of $k$ and $k[\eps]$ by $\WW(k)$ and $\WW(k[\eps])$, respectively. Then $\WW(k)$ is equal to the Witt ring $W(k)$ and there is a surjective ring homomorphism $\WW(k[\eps]) \to \WW(k)$ whose kernel consists of nilpotent elements. Let $\II_{k} = pW(k)$ (resp.~$\II_{k[\eps]}$) be the kernel of $\WW(k) \to k$ (resp.~of $\WW(k[\eps]) \to k[\eps]$). We denote the Frobenius on the Zink ring by $\sigma$ and the pull back via $\sigma$ by $(\ )^{(\sigma)}$.

The $k$-valued point $x$ is given by a tuple $(A,\lambda,\iota,\eta)$. Let $(P,Q,F,F_1)$ be the Dieudonn\'e display attached to its $p$-divisible group. It is endowed with a similitude class of perfect alternating forms $\lrangle$ induced by $\lambda$ and an $O_B$-action induced by $\iota$. In other words, $P$ is a $\Dscr$-module over $\WW(k)$. As $W(k)$ is strictly henselian, we may choose by Proposition~\ref{StackDMod} an isomorphism of $P$ with the $\Dscr$-module $\Lambda \otimes_{\ZZ_p} \WW(k)$. This defines in particular a $\ZZ_p$-rational structure on the $\Dscr$-module $P$ and an identification $P/\II_kP = \Lambda_k$.

The Hodge filtration is given by $C = Q/\II_kP \subset H^{\rm DR}_1(A/k) = P/\II_kP$. Choose a totally isotropic $O_B$-invariant direct summand $C'$ of $P$ lifting $C$ and let $T' \subset P^{(\sigma)} = P$ be an $O_B$-invariant totally isotropic complement of $(C')^{(\sigma)}$. Then the composition
\[
g\colon P \cong P^{(\sigma)} = (C')^{(\sigma)} \oplus T' \vartoover{35}{F_1\rstr{(C')^{(\sigma)}} \oplus F\rstr{T'}} P
\]
is an automorphism of $\Dscr$-modules, i.e., $g \in \Gtilde(\WW(k))$. The image of $g(T')$ in $P/\II_kP$ is the conjugate filtration $D$.

The morphism $\xtilde \colon \Spec k[\eps] \to \Bcal_{J,K}$ then corresponds to totally isotropic $O_B$-invariant direct summands $\Ctilde$ and $\Dtilde$ of $\Lambda_{k[\eps]}$ lifting $C$ and $D$, respectively. Applying Lemma~\ref{OpenBruhat} below to the stabilizers of $\Dtilde$ and $\Ctilde^{(p)}$, there exists $\gtilde \in G(k[\eps])$ with $\gtilde \equiv g \mod \eps$ such that $\Ttilde := \gtilde^{-1}(\Dtilde)$ is a complement of $\Ctilde^{(p)}$. As explained in~\cite{ViWd}~Section~4.1, triples $(C',T',g)$ as above are parametrized by a smooth scheme and thus we can apply \cite{ViWd}~Lemma~3.1 to see that there exists a totally isotropic $O_B$-invariant direct summand $\Ctilde'$ of $\Lambda_{\WW(k[\eps])}$, a totally isotropic $O_B$-invariant complement $\Ttilde'$ of $(\Ctilde')^{(\sigma)}$ and $\gtilde' \in \Gtilde(\WW(k[\eps]))$ whose pullback to $\WW(k)$ is $(C',T',g)$ and whose pullback to $k[\eps]$ is $(\Ctilde,\Ttilde,\gtilde)$. We obtain a Dieudonn\'e display $(\Ptilde,\Qtilde,\Ftilde,\Ftilde_1)$ as follows. We define $\Ptilde := \Lambda_{\WW(k[\eps])}$ and denote by $\Qtilde$ the inverse image of $\Ctilde$ under $\Ptilde \to \Ptilde/\II_{k[\eps]}\Ptilde = \Lambda_{k[\eps]}$. Finally let $\Ftilde\colon \Ptilde^{(\sigma)} \to \Ptilde$ and $\Ftilde_{1}\colon \Qtilde^{(\sigma)} \to \Ptilde$ be the unique $\WW(k[\eps])$-linear maps making $(\Ptilde,\Qtilde,\Ftilde,\Ftilde_1)$ into a Dieudonn\'e display such that the direct sum of the restriction of $\Ftilde_1$ to $(\Ctilde')^{(\sigma)}$ and the restriction of $\Ftilde$ to $\Ttilde'$ is given by $\gtilde$. Then $(\Ptilde,\Qtilde,\Ftilde,\Ftilde_1)$ is a Dieudonn\'e display with symplectic form and $O_B$-action lifting $(P,Q,F,F_1)$ and thus defining a morphism $\Spec k[\eps] \to \Ascr_0$ which 2-commutes with~\eqref{EqTangent}.
\end{proof}

\begin{example}\label{BruhatSiegel}
We consider again the Siegel case, i.e., $\BB = \QQ$ and therefore $\Gtilde = \GSp(\Lambda,\lrangle)$. Set $g := \dim_\QQ(V)/2$. We have seen in Example~\ref{Siegel} that the Bruhat stratification on $\Ascr(\kgbar)$ can also be described by
\[
\Ascr(\kgbar) = \bigcup_{0\leq i \leq m}\Acal_i,
\]
where $\Acal_i$ denotes the locally closed subvariety of $\Ascr(\kgbar)$ consisting of principally polarized abelian varieties whose $a$-number is equal to $i$. By Corollary~\ref{AStratProp} we have
\[
\overline{\Acal_i} = \bigcup_{j \geq i}\Acal_j
\]
and $\Acal_i$ is smooth and equi-dimensional of dimension $d(i)$ with
\[
d(i) := \sum_{j=i+1}^gj = \frac{g(g+1) - i(i+1)}{2}.
\]
This explicit description of $\ell(x^{J,J})$ for $x \in \doubleexp{J}{W}{J} \cong \{0,\dots,g\}$ follows from the explicit calculation of the length in \cite{ViWd}~Appendix~A.7, in particular the formulas (12.3) and (12.4) there. Note that for the identification of $\doubleexp{J}{W}{J}$ and $\{0,\dots,g\}$ made here in Example~\ref{Siegel}, the map in loc.~cit.~(12.4) should be given by $(\eps_i)_i \sends \#\set{i}{\eps_i = 0}$.
\end{example}


\section{Deforming parabolics in opposite position}\label{PAROPP}

Here we prove a group-theoretical lemma used in the proof of Theorem~\ref{ASmooth}. Let $k$ be any field, and let $G$ be a reductive group over $k$. Let $(W,I)$ be the Weyl group of $G$ together with its set of simple reflections. Recall that for any $k$-scheme $S$, two parabolic subgroups $P$ and $Q$ of $G_S$ are called \emph{opposite} if $P \cap Q$ is a Levi subgroup of $P$ and of $Q$. We indicate this by writing $P \bowtie Q$. Moreover for every Levi subgroup $L$ of a parabolic subgroup $P$ there exists a unique parabolic subgroup $Q$ of $G_S$ such that $P \cap Q = L$ and then $P \bowtie Q$ (\cite{SGA3}~Exp.XXVI~Th\'eor\`eme~4.3.2). This implies that if $J \subseteq I$ is the type of $P$, the type $J^{\rm opp}$ of a parabolic subgroup $Q$ opposite to $P$ depends only on $J$.

\begin{lemma}\label{OpenBruhat}
Let $Z_J$ be the $k$-scheme whose $S$-valued points are the set of triples $(P,Q,g)$, where $P$ is a parabolic subgroup of $G_S$ of type $J$, $Q$ is a parabolic subgroup of $G_S$ of type $J^{\rm opp}$, and $g \in G(S)$ such that $\leftexp{g}{Q} \bowtie P$. Then the canonical morphism
\[
Z_J \to \Par_J \times \Par_{J^{\rm opp}}, \qquad (P,Q,g) \sends (P,Q)
\]
is smooth.
\end{lemma}

\begin{proof}
We set $X := \Par_J \times \Par_{J^{\rm opp}}$. Let $\Pcal_J \subset G_{\Par_J}$ be the universal parabolic subgroup of $G_{\Par_J}$ of type $J$ and let $\Pcal'_J$ be its pull back to $X$ under the first projection. This is a parabolic subgroup of $G_X$. Define similarly a parabolic subgroup $\Pcal'_{J^{\rm opp}}$ of $G_X$. Then the $X$-scheme $Z_J$ is isomorphic to the subscheme of $G_X$ whose $S$-valued points are given by $\set{g \in G_X(S)}{\leftexp{g}{\Pcal'_{J^{\rm opp}}} \bowtie \Pcal'_J}$. This is an open subscheme of $G_X$ by \cite{SGA3}~Exp.XXVI~Th\'eor\`eme~4.3.2. As $G_X$ is smooth over $X$, this implies that $Z_J$ is smooth over $X$.
\end{proof}


\bibliography{references}

\end{document}